\documentclass[11pt]{article}

\usepackage{amsmath,amsthm,verbatim,amssymb,amsfonts,amscd, graphicx,mathrsfs}

\usepackage{ upgreek }

\usepackage{graphics}

\theoremstyle{plain}

\newtheorem{theorem}{Theorem}

\newtheorem{proposition}{Proposition}

\theoremstyle{definition}

\newtheorem{definition}{Definition}

\newcommand{\FF}{\mathcal{F}}

\newcommand{\GG}{\mathcal{G}}

\begin{document}
\title{On the $d$-cluster generalization of Erd\H{o}s-Ko-Rado}
\author{Gabriel Currier \footnote{Department of Mathematics, University of British Columbia, Vancouver, Canada. currierg@math.ubc.ca \newline \indent \indent Keywords: Set-system, Erd\H{o}s-Ko-Rado, Cluster}}
\maketitle
\begin{abstract}
\noindent If $2 \le d \le k$ and $n \ge dk/(d-1)$, a $d$-cluster is defined to be a collection of $d$ elements of ${[n] \choose k}$ with empty intersection and union of size no more than $2k$. Mubayi  conjectured that the largest size of a $d$-cluster-free family $\FF \subset {[n] \choose k}$ is ${n-1 \choose k-1}$, with equality holding only for a maximum-sized star. Here, we resolve Mubayi's conjecture and prove a slightly stronger result, thus completing a new generalization of the Erd\H{o}s-Ko-Rado Theorem.
\end{abstract}
\section{Introduction}
For any $m,n \in \mathbb{Z}$ with $m < n$, we define $[m,n] := \{m,m+1,\dots,n\}$, and $[n] := [1,n]$, and we use ${X \choose k}$ to refer to the set of $k$-element subsets of a set $X$. Furthermore, if every element of a family $\FF \subset {X \choose k}$ contains some $x \in X$, we say that $\FF$ is a star (centered at $x$). We note that the maximum size of a star in ${[n] \choose k}$ is ${n-1 \choose k-1}$ and recall the classical Erd\H{o}s-Ko-Rado (EKR) theorem, which gives an upper bound on the size of ``pairwise intersecting'' set systems.
\begin{theorem}[\cite{ekr}]\label{oldekr}
Let $n \ge 2k$ and suppose that $\FF \subset {[n] \choose k}$ has the property that $A \cap B \neq \emptyset$ for all $A,B \in \FF$. Then $$|\FF| \le {n-1 \choose k-1},$$ where, excluding the case $n = 2k$, equality is achieved only when $\FF$ is a maximum-sized star.
\end{theorem}
\noindent Theorem \ref{oldekr} is one of the fundamental results in extremal combinatorics, and has spawned numerous generalizations and conjectures. In particular, it has generated a whole area of ``intersection problems'' for set systems, in which we consider the maximum size of a family where we have forbidden a certain class of subfamilies, defined according to some intersection and union constraints. One example of such a class which is directly related to the original EKR theorem is the notion of a $d$-cluster, introduced by Mubayi in \cite{3mub}.

\begin{definition}
Let $2 \le d \le k$ with $n \ge dk/(d-1)$ and suppose that $\FF \subset {[n] \choose k}$. Then, if we have $\mathcal{B} = \{B_1,\dots,B_d\} \subset {[n] \choose k}$ such that $|B_1 \cup \dots \cup B_d| \le 2k$ and $B_1 \cap \dots \cap B_d = \emptyset$, we say that $\mathcal{B}$ is a \emph{$d$-cluster.} Furthermore, if $\FF$ contains no such $\mathcal{B}$, we say that $\FF$ is \emph{$d$-cluster-free.}
\end{definition}

\noindent Note here that in the case of $d=2$, the union condition holds automatically. Thus, another way to state Theorem \ref{oldekr} would be to say that $2$-cluster-free families can have size no greater than ${n-1 \choose k-1}$. In \cite{3mub}, Mubayi showed that $3$-cluster-free families must also obey this bound, and conjectured that the same would hold for $d \ge 4$. The primary goal of this paper is to resolve Mubayi's conjecture.
\begin{theorem}\label{mubconj}
Let $2 \le d \le k$ and $n \ge dk/(d-1)$. Furthermore, suppose that $\FF \subset {[n] \choose k}$ is $d$-cluster-free. Then 
$$|\FF| \le {n-1 \choose k-1},$$
and, excepting the case when both $d=2$ and $n = 2k$, equality implies $\FF$ is a maximum-sized star.
\end{theorem}

\noindent We note that the conditions $d\le k$ and $n \ge dk/(d-1)$ are in fact necessary; if we let $d = k+1$ then the problem is equivalent to one of the long-open hypergraph Tur\'an problems, and a different bound applies. Furthermore, if $n < dk/(d-1)$, then any collection of $d$ sets cannot have empty intersection. A number of previous results on this problem have been shown; Katona first proposed the problem as a version of the $d=3$ case, and Frankl and F\"{u}redi in \cite{franfur} obtained the desired bound for $d = 3$ and $n \ge k^2 + 3k$. Mubayi, in addition to resolving the $d=3$ case and proposing the more general problem in \cite{3mub}, showed in \cite{4mub} that the theorem holds when $d = 4$ and $n$ is sufficiently large. Later, Mubayi and Ramadurai \cite{rammub} and independently \"{O}zkahya and F\"{u}redi \cite{ozkfur} showed that it also holds when $d > 4$ and $n$ is sufficiently large. In \cite{kevmub}, Keevash and Mubayi solved another case of this problem, namely where both $k/n$ and $n/2-k$ are bounded away from zero. The case of $n < 2k$ (where, again, the union condition holds automatically) was resolved by Frankl in \cite{spernfam} (where the bound was established) and in \cite{shiftfrank} (where equality was characterized). 
\\\\
We note finally that the problem of $d$-clusters is closely related to several other old problems in extremal set theory, in particular the simplex and special simplex conjectures of Chv\'atal \cite{chva} and Frankl and F\"uredi \cite{franfur2}, and more generally the so-called Tur\'an problems for expansion. For a detailed historical account of these kinds of problems, as well as some exciting new results, we direct the interested reader to \cite{kellif}. We also note that most remaining cases of Chv\'atal's conjecture were shown in a recent paper of the author \cite{me}, in which an alternative (but related) proof of Theorem \ref{mubconj} for a more limited range of parameters is also given. However, the present manuscript, and the techniques described here, appeared first.
\\\\
\noindent Despite the success in tackling the problem of clusters for large values of $n$, techniques that could obtain the desired bound for $d \ge 4$ and all $n \ge 2k$ have remained largely elusive. In this paper, we bridge the gap by giving a slightly more general result that is valid for all $d \ge 2$ and $n \ge 2k$, which in combination with the known results on $n < 2k$ will give us Theorem \ref{mubconj}. Our more general theorem will give us more powerful inductive tools that we will leverage in our proof, and is based on a method of induction introduced in \cite{4mub}. 

\begin{theorem}\label{mainthm}
Let $2 \le d \le k \le n/2$.  Furthermore, suppose that $\FF \subset {[n] \choose k}$, and that $\FF^* \subset \FF$ has the property that any $d$-cluster in $\FF$ is contained entirely in $\FF^*$. Then, 
$$|\FF^*| + \frac{n}{k}|\FF \setminus \FF^*| \le {n \choose k}.$$
Furthermore, excepting the case where both $d = 2$ and $n = 2k$, equality implies one of the following:
\begin{enumerate}
\item $\FF^* = \emptyset$ and $\FF$ is a maximum-sized star,
\item $\FF = \FF^* = {[n] \choose k}$.
\end{enumerate}
\end{theorem}

\noindent Note that Theorem \ref{mainthm} reduces to Theorem \ref{mubconj} for all $n \ge 2k$ if we set $\FF^* = \emptyset$. The $d=2$ case of Theorem \ref{mainthm} is itself an interesting strengthening of Erd\H os-Ko-Rado, and has actually appeared before in the literature. To our knowledge, it was shown first by Borg in \cite{borg1} using shadow techniques, and by Borg and Leader in \cite{borg2} using cycle methods, and can also be seen as a consequence of a more general result in \cite{frankup} (Theorem $9$ with $c = \frac{n-k}{k}$). These results are all related to \emph{cross-intersecting} families; that is, families $\FF,\GG \subset {[n] \choose k}$ such that for all $A \in \FF$ and $B \in \GG$ we have $A \cap B \neq \emptyset$. In fact, the condition required for Theorem \ref{mainthm} to hold (in the $d=2$ case) is equivalent to requiring that $\FF$ and $\FF \setminus \FF^*$ are cross-intersecting. The connection between Theorem \ref{mainthm} and the notion of cross-intersecting families is an interesting one and perhaps deserves further attention.

\section{Proof of Theorem \ref{mainthm}}

\noindent Before we proceed with the proof of Theorem \ref{mainthm}, we will need some auxiliary results and will use the following notation. For any $\FF \subset {[n] \choose k}$ and $D \subset [n]$, we let

$$\bigtriangledown_\FF(D) := \{B \setminus D :  B \in \FF \text{ and } D \subset B\}.$$

\noindent Furthermore, when $D = \{x\}$, we write simply $\bigtriangledown_\FF(x)$. Our proof will proceed by induction on $d$, and the following result will allow us to perform this induction. It is a stronger version of a proposition from \cite{4mub} that was later stated more clearly in \cite{rammub}.

\begin{proposition}\label{trianglepart}
Let $3 \le d \le k$ and $n \ge dk/(d-1)$. Furthermore, suppose $\FF^* \subset \FF \subset {[n] \choose k}$ has the property that any $d$-cluster in $\FF$ is contained in $\FF^*$. Then, if $\{D_1,\dots,D_{d-1}\} \subset \bigtriangledown_\FF(x)$ is a $(d-1)$ cluster, it follows that for each $i \in [d-1]$, we have $|\bigtriangledown_\FF(D_i)| = 1$ or $D_i \in \bigtriangledown_{\FF^*}(x)$.
\end{proposition}

\begin{proof}
Suppose for the sake of contradiction that there exists $i_0 \in [d-1]$ such that $|\bigtriangledown_\FF(D_{i_0})| \ge 2$ and $D_{i_0} \notin \bigtriangledown_{\FF^*}(x)$. Then, we know both that $(D_{i_0} \cup \{x\}) \in (\FF \setminus \FF^*)$ and that $(D_{i_0} \cup \{y\}) \in \FF$ for some $y \in [n] \setminus \{x\}$. Then
$$|(D_1 \cup \{x\}) \cup (D_2 \cup \{x\}) \cup \dots \cup (D_{d-1} \cup \{x\}) \cup (D_{i_0} \cup \{y\})| \le |D_1 \cup \dots \cup D_{d-1}| + |\{x,y\}| \le 2(k-1) + 2 = 2k,$$
and since, $x,y \notin D_{i_0}$, we see
$$(D_1 \cup\{x\}) \cap \dots \cap (D_{d-1} \cup \{x\}) \cap (D_{i_0} \cup \{y\}) \subset D_1 \cap \dots \cap D_{d-1} = \emptyset.$$
Furthermore, since $x \notin D_{i_0}$, we know that $(D_{i_0} \cup \{y\}) \neq (D_j \cup \{x\})$ for any $j \in [d-1]$. Thus, $\FF$ contains a $d$-cluster not contained entirely in $\FF^*$, which is a contradiction.

\end{proof}
\noindent The following proposition, furthermore, will assist us in a proof of Theorem \ref{mainthm}.
\begin{proposition}\label{ineq}
Let $2 \le k < n$ and $\FF \subset {[n] \choose k}$. Then, the following hold:
\begin{enumerate}
\item $\sum_{x \in [n]} |\bigtriangledown_\FF(x)| = k|\FF|$,
\item $|\{D \in {[n] \choose k-1} \ : \ |\bigtriangledown_\FF(D)| = 1\}| \le \frac{n{n-1 \choose k-1} - k|\FF|}{n-k}$.
\end{enumerate}
\end{proposition}

\begin{proof}
The proof of $(i)$ is straightforward from the definitions, so we will focus on $(ii)$. First, we denote by $\FF^C$ the complement of $\FF$ in ${[n] \choose k}$, and observe that 
$$\sum_{x \in [n]} (|\bigtriangledown_\FF(x)| + |\bigtriangledown_{\FF^C}(x)|) =  n{n-1 \choose k-1}.$$
Using this along with $(i)$, we see
\begin{align*}
|\{D \in {[n] \choose k-1} : |\bigtriangledown_\FF(D)| = 1\}| &= |\{D \in {[n] \choose k-1} : |\bigtriangledown_{\FF^C}(D)| = n-k\}| \\
&\le \frac{\sum_{x\in[n]}|\bigtriangledown_{\FF^C}(x)|}{n-k}\\
&=\frac{n{n-1 \choose k-1} - \sum_{x\in[n]}|\bigtriangledown_\FF(x)|}{n-k} \\
&= \frac{n{n-1 \choose k-1} - k|\FF|}{n-k}.
\end{align*}
\end{proof}
\noindent We now begin with the proof of our main result.

\begin{proof}[Proof of Theorem \ref{mainthm}]
Let $\FF^* \subset \FF$ be as described. We note as before that the $d=2$ case appears already in the literature, so we suppose $d \ge 3$. For any $x \in [n]$, we let
$$\bigtriangledown^*_{\FF}(x) := \{D \in \bigtriangledown_\FF(x) : |\bigtriangledown_\FF(D)| = 1\} \cup \bigtriangledown_{\FF^*}(x).$$
Observe first that the sets $\{D \in \bigtriangledown_\FF(x) : |\bigtriangledown_\FF(D)| = 1\}$ over $x \in [n]$ partition $\{D \in {[n] \choose k-1} : |\bigtriangledown_\FF(D)| = 1\}$. Using this and Proposition \ref{ineq} yields
\begin{align}
\sum_{x \in [n]}|\bigtriangledown^*_{\FF}(x)| &\le |\{D \in {[n] \choose k-1} : |\bigtriangledown_\FF(D)| = 1\}| + \sum_{x\in[n]} |\bigtriangledown_{\FF^*}(x)| \nonumber \\
&\le  \frac{n{n-1 \choose k-1} - k|\FF|}{n-k} + k|\FF^*| \label{startri}.
\end{align}
Furthermore, by Proposition \ref{trianglepart}, we know that any $(d-1)$-cluster in $\bigtriangledown_\FF(x)$ is contained in $\bigtriangledown^*_{\FF}(x)$. Since $\bigtriangledown_\FF(x) \subset {[n] \setminus \{x\} \choose k-1}$ and $(d-1) \le (k-1) < (n-1)/2$, we may apply induction on $d$ to get
\begin{align}
{n-1 \choose k-1} \ge |\bigtriangledown^*_{\FF}(x)| + \frac{n-1}{k-1}|\bigtriangledown_\FF(x) \setminus \bigtriangledown^*_{\FF}(x)| = \frac{n-1}{k-1}|\bigtriangledown_\FF(x)| - \frac{n-k}{k-1}|\bigtriangledown^*_{\FF}(x)|.\label{xline}
\end{align}
Then, summing over all $x \in [n]$ and using Proposition \ref{ineq} in combination with (\ref{startri}), we see

\begin{align*}
n{n-1 \choose k-1} &\ge \frac{n-1}{k-1}\sum_{x \in [n]}|\bigtriangledown_\FF(x)| - \frac{n-k}{k-1}\sum_{x \in [n]}|\bigtriangledown^*_{\FF}(x)|\\
&\ge \frac{n-1}{k-1}k|\FF| - \frac{n-k}{k-1}k|\FF^*| - \Big(\frac{n-k}{k-1}\Big)\frac{n{n-1 \choose k-1} - k|\FF|}{n-k} \\
&=\frac{nk}{k-1}|\FF| - \frac{(n-k)k}{k-1}|\FF^*| - \frac{n{n-1 \choose k-1}}{k-1},
\end{align*}
and therefore
\begin{align*}
\frac{nk}{k-1}{n-1 \choose k-1} \ge \frac{nk}{k-1}|\FF| - \frac{(n-k)k}{k-1}|\FF^*|.
\end{align*}
Finally, multiplying both sides by $\frac{k-1}{k^2}$ gives us
\begin{align*}
{n \choose k} \ge \frac{n}{k}|\FF| -\frac{n-k}{k}|\FF^*| = |\FF^*| + \frac{n}{k}|\FF \setminus \FF^*|.
\end{align*}
Suppose now that we have equality - that is, that $|\FF^*| + (\frac{n}{k})|\FF \setminus \FF^*| = {n \choose k}$. This implies two things. First, we get that $|\FF| \ge {n-1 \choose k-1}$ with $|\FF| = {n-1 \choose k-1}$ only if $\FF^* = \emptyset$. Additionally, we must have equality in (\ref{xline}) for all $x \in [n]$, that is, $$|\bigtriangledown^*_{\FF}(x)| + \frac{n-1}{k-1}|\bigtriangledown_\FF(x) \setminus \bigtriangledown^*_{\FF}(x)| = {n-1 \choose k-1}.$$
Furthermore, since $(n-1) > 2(k-1)$, we get by induction that either $\bigtriangledown^*_{\FF}(x) = \bigtriangledown_\FF(x)= {[n] \setminus \{x\} \choose k-1}$ or $\bigtriangledown_\FF(x)$ is a maximum sized star and $\bigtriangledown^*_{\FF}(x) = \emptyset$. Suppose for the sake of contradiction that the latter is true for all $x \in [n]$. Then $|\bigtriangledown_\FF(x)| = {n-2 \choose k-2}$  for all  $x \in [n]$, and using proposition \ref{ineq} yields
$$|\FF| = \frac{\sum_{x \in [n]}|\bigtriangledown_\FF(x)|}{k} = \frac{n}{k}{n-2 \choose k-2} < \frac{n-1}{k-1}{n-2 \choose k-2} = {n-1 \choose k-1},$$
which is a contradiction. Thus, there exists $x_0 \in [n]$ such that $\bigtriangledown_{\FF}(x_0) = {[n] \setminus \{x_0\} \choose k-1}$, and therefore $\FF$ contains a maximum-sized star centered at $x_0$. If $|\FF| = {n-1 \choose k-1}$ this implies that  $\FF$ is, itself, a maximum size star centered at $x_0$ and that $\FF^* = \emptyset$. Now, suppose $|\FF| >{n-1 \choose k-1}$ and take $B,B' \in \FF$ such that exactly one of $B,B'$ contains $x_0$ (note that there must exist at least one element of $\FF$ not containing $x_0$). Furthermore, take $Z \subseteq [n]$ such that $|Z| = 2k$ and $(B \cup B') \subset Z$. Then, since $|B \cap B'|  \le k-1$ we know that $|Z \setminus \{x_0\} \setminus (B \cap B')| \ge k$ and thus there exist distinct
$$D_1,\dots,D_{d-2} \in {Z \setminus \{x_0\} \setminus (B \cap B') \choose k-1}$$
such that $(D_i \cup \{x_0\}) \neq B,B'$ for all $i \in [d-2]$. Furthermore, since $\FF$ contains a maximum-sized star centered at $x_0$, we get $(D_i \cup \{x_0\}) \in \FF$ for all $i \in [d-2]$ and
$$|B \cup B' \cup (D_1 \cup \{x_0\}) \cup \dots \cup (D_{d-2} \cup \{x_0\})| \le |Z| = 2k.$$
Additionally, because $x_0$ is not in one of $B$ or $B'$, we see
$$B \cap B' \cap (D_1 \cup \{x_0\}) \cap \dots \cap (D_{d-2} \cup \{x_0\}) = \emptyset.$$
Thus, every element of $\FF$ is part of a $d$-cluster. Since all $d$-clusters in $\FF$ are contained in $\FF^*$, we get that $\FF = \FF^*$ and thus that $|\FF| = {n \choose k}$ and $\FF = {[n] \choose k}$. This completes the proof.
\end{proof}

\noindent {\bfseries Acknowledgements:} I would like to thank my professors Dr. Shahriar Shahriari and Dr. Ghassan Sarkis as well as the Pomona College Research Circle for first introducing me to this problem, and working on it with me in the early stages. In particular, I would like to thank Archer Wheeler for continuing to work with me on this project (in its various forms) as well as Dr. Shahriar Shahriari, Dr. Richard Anstee, and Dr. Jozsef Solymosi for their support and helpful comments. Finally, I am very grateful to the referees for a careful reading, helpful comments, and a number of additional references (in particular concerning the $d=2$ case of Theorem \ref{mainthm}).

\bibliographystyle{amsplain}
\bibliography{set_bib}
\end{document}